\documentclass[11pt]{article}
\usepackage[left=1in,right=1in,bottom=1in,top=1in]{geometry}
\usepackage{amsmath}
\usepackage{amsfonts}
\usepackage{amssymb}
\usepackage{amsthm}
\usepackage{young}
\usepackage[vcentermath]{youngtab}
\usepackage{tikz}
\usetikzlibrary{shapes.misc}

\tikzset{cross/.style={cross out, draw=black, minimum size=2.5*(#1-\pgflinewidth), inner sep=2pt, outer sep=0.5pt},
	cross/.default={1pt}}
\usepackage{xcolor}
\usetikzlibrary{calc,arrows}

\usepackage{algorithm}
\usepackage[noend]{algpseudocode}
\usepackage{amsmath, amssymb, amsthm}
\usepackage{blindtext}
\usepackage[pdftex,bookmarks=true]{hyperref}
\usepackage{cleveref}
\usepackage{comment}
\usepackage{enumerate}
\usepackage{fullpage}
\usepackage{mathtools}
\usepackage{subcaption}

\usepackage{enumitem}
\setlist[enumerate]{itemsep=.3mm}
\setlist[itemize]{itemsep=.3mm}

\newcommand{\E}{\mathbb E}
\renewcommand{\P}{\mathbb P}
\newcommand{\Bin}{\textbf{Bin}}
\newcommand{\ep}{\varepsilon}
\renewcommand{\deg}{\text{deg}}
\newcommand{\one}{\textnormal{\textbf{1}}}
\newcommand{\zero}{\textnormal{\textbf{0}}}

\newtheorem{theorem}{Theorem}[section]
\newtheorem{lemma}[theorem]{Lemma}

\newtheorem{proposition}[theorem]{Proposition}

\newtheorem{thm}{Theorem}[section]

\newtheorem{cor}[thm]{Corollary}
\newtheorem{prop}[thm]{Proposition}

\theoremstyle{definition}
\newtheorem{definition}[theorem]{Definition}
\newtheorem{remark}[theorem]{Remark}
\newtheorem*{remark-non}{Remark}

\title{\textbf{A discontinuous phase transition in the threshold-$\theta \geq 2$ contact process on random graphs}}

\author{Danny Nam\footnote{Department of mathematics, Princeton University; dhnam@princeton.edu.} \footnote{This work is a result of 2019 AMS MRC: Stochastic spatial models.} }
\date{\vspace{-5ex}}

\begin{document}
	\maketitle
	\begin{abstract}
		We study the discrete-time threshold-$\theta \geq 2$ contact process on random graphs of general degrees. For random graphs with a given degree distribution $\mu$, we show that if $\mu$ is lower bounded by $\theta+2$ and has finite $k$th moments for all $k>0$, then the discrete-time threshold-$\theta$ contact process on the random graph exhibits a discontinuous phase transition in the emergence of metastability, thus answering a question of Chatterjee and Durrett \cite{cd13}. To be specific, we establish that (i) for  any large enough infection probability $p>p_1$, the process started from the all-infected state \textsf{whp} survives for  $e^{\Theta(n)}$-time, maintaining a large density  of infection; (ii) for any $p<1$, if the initial density is smaller than $\varepsilon(p)>0$, then it dies out in $O(\log n)$-time \textsf{whp}. We also explain some extensions  to more general random graphs, including the Erd\H{o}s-R\'enyi graphs. Moreover, we prove that the threshold-$\theta$ contact process on a random $(\theta+1)$-regular graph dies out in time $n^{O(1)}$ \textsf{whp}. 
	\end{abstract}

\section{Introduction}

In this paper, we study the discrete-time threshold-$\theta$ contact process with $\theta\geq 2$. The model is defined on a graph $G=(V,E)$, and the configuration of the process at time $t \in \mathbb{N}$ is $X_t \in \{0,1 \}^V$, where 0 and 1 denote `healthy' and `infected' individuals, respectively. Suppose that the state $X_t$ at time $t$ is given. Then its transitions are defined as follows.

\begin{itemize}
	\item [$\bullet$] If a vertex $v$ has at least $\theta$ neighbors $u$ with $X_t(u)=1$, then $X_{t+1}(v)=1$ with probability $p$ and $X_{t+1}(v)=0$ with probability $1-p$. 
	
	\item [$\bullet$] If $v$ has less than $\theta$ infected neighbors at time $t$, then $X_{t+1}(v)=0$ with probability $1$.
	
		\item [$\bullet$] At each time step, transitions happen independently and simultaneously at all sites.
\end{itemize}

Regarding $0$ and $1$ as `vacant' and `occupied by a particle,' the process can be thought of a population growth model with sexual reproduction \cite{t74,bg911,c92,c94,dg85,ms09}. It can also be seen as a dynamical version of the bootstrap percolation process (see, e.g., \cite{al03} for a review on this topic).

One interesting aspect of the threshold contact process is that on various graphs including lattices, (homogeneous or Galton-Watson) trees and random graphs, the process is predicted to display a discontinuous phase transition. To be specific, on infinite graphs, we expect the density  of infection in the upper invariant measure (i.e., the law of the system at $t\rightarrow \infty$ limit started from all-infected initial state) to be discontinuous at $p_c$, the point where  a nonzero invariant measure emerges. On (finite but large) random graphs, we believe that there are  metastable states (i.e., configurations that survive for exponentially long time) at $p>p_c$ and  their densities stay strictly away from $0$ as $p \searrow p_c$. However, not much is known in a rigorous sense. 

Chatterjee and Durrett \cite{cd13} studied this model on random $d$-regular graphs and infinite $d$-regular trees, and showed that for any $d\geq \theta +2$, there is such a discontinuous phase transition. Further, they asked if this  is true for random graphs with a general degree distribution $\mu$. Our main result answers this question for $\mu$ with a light enough tail. Denoting the random graph with degree distribution $\mu$ by $G_n \sim \mathcal{G}(n,\mu)$ (see Section \ref{subsec:prelim} for its full definition), the main theorem can be stated as follows.

\begin{thm}\label{thm:thm1}
	Let $\theta\geq 2$ be an integer and $\mu$ be a distribution on $\mathbb{N}$ such that $D\sim\mu$ satisfies $\P(D\geq \theta+2)=1$ and $ \E D^k <\infty$ for all $k>0$.  \textsf{Whp} over the choice of $G_n \sim \mathcal{G}(n,\mu)$, the discrete-time threshold-$\theta$ contact process $(X_t)$ on $G_n$ satisfies the following:
	
	\begin{enumerate}
		\item[{\textnormal 1.}] There exist $p_1 \in (0,1)$ and $\varepsilon_1,\gamma_1>0$ such that if $p\geq p_1$ and $|X_0|\geq (1-\varepsilon_1)n$, then $|X_t|\geq (1-\varepsilon_1)n$ for all $t\in[0, e^{\gamma_1 n}]$ \textsf{whp}.
		
		\item [{\textnormal 2.}] For any $p\in(0,1)$, there exist constants  $\varepsilon_2(p), C_2(p)>0$ such that if $|X_0|\leq \varepsilon_2 n$, then $X_{C_2 \log n} = \textbf{0}$ \textsf{whp}. 
	\end{enumerate}
\end{thm}

\noindent In the statement, \textsf{whp} stands for \textsf{with high probability}, meaning that an event happens with probability tending to $1$ as $n\rightarrow \infty$. Also, note that there are two layers of randomness: we first generate a random graph $G_n$, and then on (fixed) $G_n$ we run the threshold-$\theta$ contact process which is a random process. In Section \ref{subsec:generalized results}, we discuss some generalizations of Theorem \ref{thm:thm1}.

Theorem \ref{thm:thm1} implies that the threshold-$\theta$ contact process on $G_n\sim \mathcal{G}(n,\mu)$ exhibits a discontinuous phase transition: along with the monotonicity of our model (in terms of $p$),  the first statement shows that there is the regime  $p>p_c$ where a metastable state  emerges, and the second tells us that its density should not be too close to $0$. The second part also implies that $p_c >0$, for instance by setting $p<\varepsilon_2(\frac{1}{2})$, we get $|X_1| \leq \varepsilon_2 n$ and then a logarithmic time survival.

Intuitively, we have a fairly clear reason for a discontinuous phase  transition to take place, which we briefly describe as follows.
\vspace{2mm}

\noindent $\blacktriangleright$ \textit{Mean-field calculation and  metastability.}  ~Consider the threshold-$\theta$ contact process on a $(\theta+2)$-regular tree, and suppose that we are interested in the flow of infection towards the root. Let $v$ be a vertex other than the root and assume that at time $t$ each child $u\sim v$ is infected with probability $q$ independently of each other. Then, by only looking at the infections among its children, the probability that $v$ is infected at time $t+1$ is at least
\begin{equation*}
(\theta+1) pq^{\theta} (1-q) + pq^{\theta+1} = pq^{\theta}((\theta+1)- \theta q) =: f(q),
\end{equation*}
and one can see that for $p$ close enough to $1$, there is a stable fixed point of $f(q)=q$. This implies that the threshold-$\theta$ contact process on infinite $(\theta+2)$-regular tree should survive for large enough $p$ and hence have a nontrivial stationary distribution. Since we assumed that $\mu$ is lower bounded by $\theta+2$, the local neighborhoods of $G_n \sim \mathcal{G}(n,\mu)$ roughly dominate $(\theta+2)$-regular trees, inferring the existence of metastability on $G_n$ for large $p$. 

\vspace{2mm}
\noindent $\blacktriangleright$ \textit{Extinction of a low-density state.} ~Suppose that the local neighborhood $N(v,R+1)$ of $v\in G_n$ is a tree for some constant $R$. In such a case, it is clear that if $X_0 = \one_{N(v,R)}$, then $X_t(u)=0$ for all $t$ and $u\notin N(v,R)$. This shows that the infections inside a `bounded treelike region' cannot percolate, and hence they may die out rapidly. The infections inside a multiple of such regions that are distant from each other  have the same qualitative behavior, inferring extinction started from a  state with very low density. To justify this intuition, one has to consider all possible configurations of small density, as stated in Theorem \ref{thm:thm1}.
 
 \vspace{2mm}
 Fontes and Schonmann \cite{fs08} studied the continuous-time threshold-$\theta\geq 2$ contact process on infinite $d$-regular trees. They showed that for large $d$, the process exhibits a discontinuous phase transition by a mean-field analysis. Later, Chatterjee and Durrett \cite{cd13} extended the result to all $d\geq \theta+2$ for the discrete-time model, as a result of establishing the same type of conclusion on random $d$-regular graphs. The main contribution of this work is to demonstrate an analogous phenomenon on general random graphs. For a more detailed review on related subjects, we refer the reader to \cite{cd13} which has a nice summary on the literature.

On the other hand,  \cite{cd13} conjectured that the threshold-$\theta$ contact process should die out on random $(\theta+1)$-regular graphs, implying that the condition $D\geq \theta+2$ in Theorem \ref{thm:thm1} is not just a technical device. We verify this conjecture by establishing a polynomial upper bound on the survival time.

\begin{thm}\label{thm:thm2}
	Let $p\in(0,1)$ be a given number and $\theta \geq 2$ be an integer. There exists a constant $C_p$ such that  \textsf{whp} over the choice of $G_n \sim \mathcal{G}(n,\theta+1)$,  the threshold-$\theta$ contact process on $G_n$ with all-infected initial state dies out in time $n^{C_p}$ \textsf{whp}.
\end{thm}

In order to  prove Theorem \ref{thm:thm2}, consider a cycle $C$ inside $G_n\sim \mathcal{G}(n,\theta+1)$. Note that
\begin{center}
	if $X_t(C) = \zero$, then $X_s(C) = \zero$ for all $s\geq t$,
\end{center}
due to the definition of the threshold-$\theta$ contact process. Further, at each time step, the probability of the entire cycle $C$ becoming healthy is at least $(1-p)^{|C|}$. Therefore, we show that the graph $G_n$ can be covered with cycles of length $O(\log n)$ and then convert each of them to be all-healthy in time $(1-p)^{-O(\log n)}$, thus obtaining that the process can survive at most polynomially long time. Details are discussed in Section \ref{sec:thm2}.

%The main contribution of this paper is to develop  machinaries that make the above intuitions  rigorous on random graphs beyond the regular ones. In the following subsection, we introduce some extensions of Theorem \ref{thm:thm1} that holds for more general classes of $\mu$. 

\subsection{Results for more general random graphs}\label{subsec:generalized results}

It turns out that the condition $D\geq \theta+2$ can easily be removed if we know the existence of the `($\theta+2$)-core' inside $G_n$, the largest induced subgraph with minimal degree $\theta+2$. Kim \cite{k06} introduced a robust technique  that finds cores inside Erd\H{o}s-R\'enyi random graphs, which was later generalized to the case of $G_n \sim \mathcal{G}(n,\mu)$ in \cite{jl07}. 

Let $D\sim \mu$, $h\in (0,1)$ and $D_h = \Bin(D, h)$, where $\Bin$ is a shorthand for a binomial random variable. That is, for each $r\in \mathbb{N}$,
\begin{equation*}
\P(D_h = l) = \sum_{j\geq l} \P(\Bin(j, h)= l) \,\P(D=j).
\end{equation*}

\begin{prop}[\cite{k06,jl07}]\label{prop:kcore}
	Let $r$ be a positive integer and $\mu$ be a degree distribution with $d=\E_{D\sim \mu}D$. Define the functions $F_r=F_r(h)$ and $\rho_r =\rho_r(h) $ by
	\begin{equation}\label{eq:rcore eq def}
	\begin{split}
	F_r(h) &= \E\left[D_h\one_{\{ D_h\geq r \}} \right] = \sum_{l\geq r} \sum_{j\geq l} l\, \P( \textnormal{\Bin} (j, h)= l) \,\P(D=j);\\
	\rho_r(h) &= \P (D_h \geq r) =\sum_{l\geq r} \sum_{j\geq l} \P( \textnormal{\Bin} (j, h)= l) \,\P(D=j).
	\end{split}
	\end{equation}
	If there exists $h\in(0,1)$ such that
	\begin{equation}\label{eq:rcore condition}
	d h^2 < F_r(h),  
	\end{equation}
	then \textsf{whp}, there exists the $r$-core $K_n$ inside $G_n \sim \mathcal{G}(n,\mu)$, whose size tends to  $|K_n|/n \rightarrow \rho_r(\hat{h})$ in probability, where $\hat{h} $ is the largest $h\leq 1$ such that $dh^2 = F_r(h)$.
\end{prop}

\begin{remark}
	For Erd\H{o}s-R\'enyi random graphs $\mathcal{G}_{\textsc{er}}(n, \frac{d}{n})$, the condition (\ref{eq:rcore condition}) is equivalent to 
	\begin{equation}\label{eq:rcore ER}
	d>d_r:=\min\left\{\frac{\alpha}{\P(\textnormal{Pois}(\alpha) \geq r-1 )} \,:\, \alpha >0 \right\}.
	\end{equation}
\end{remark}

Furthermore, Theorem \ref{thm:thm1} assumed $ \mu$ to have finite $k$-th moments for all $k$. Our method generalizes to $\mu$ with a finite $K$-th moment for some large $K$. To introduce this extension, we first define $\mathcal{M}(k,M)$ to be the collection of probability distributions $\mu$ on $\mathbb{N}$ that satisfies $\E_{D\sim\mu} D^k \leq M$.

\begin{thm}\label{thm:thm4}
	Let $\theta\geq 2$ be an integer, set $k_1=6$ and let $M>0$ be a fixed constant. There exists a constant $K=K(k_1,M)$ such that for all $\mu \in \mathcal{M}(k_1,M)$  satisfying $\E_{D\sim \mu} D^{K} <\infty $ and $(\ref{eq:rcore condition})$ with $r=\theta+2$, the following holds true: there exists $p\in(0,1)$ such that \textsf{whp} over the choice of $G_n\sim \mathcal{G}(n,\mu)$, the threshold-$\theta$ contact process $(X_t)$ with probability parameter $p$ satisfies that
	\begin{enumerate}
		\item [{\textnormal 1.}]  There exist $\varepsilon_1,\gamma_1>0$ such that if $X_0 \equiv \one$, then $|X_t|\geq (1-\varepsilon_1) \rho_{\theta+2}(\hat{h}) n$ for all $t\in[0,e^{\gamma_1 n}]$ \textsf{whp}, where $\rho_{\theta+2}$ and $\hat{h}$ are as in Proposition \ref{prop:kcore}.
		\item [{\textnormal 2.}] There exist $\varepsilon_2, C_2>0$ such that if $|X_0|\leq \varepsilon_2 n$, then $X_{C_2 \log n} \equiv \textbf{0}$ \textsf{whp}.
	\end{enumerate}
\end{thm}

The Erd\H{o}s-R\'enyi random graph $\mathcal{G}_{\textsc{er}}(n, \frac{d}{n})$ is contiguous to $\mathcal{G}(n,\mu)$ with $\mu = \textnormal{Pois}(d)$ in the sense that for an event $\mathcal{A}_n$,
$$\lim_{n\rightarrow \infty}\P_{G_n\sim \mathcal{G}(n,\textnormal{Pois}(d)) } (G_n \in \mathcal{A}_n) 
= 0 \quad \textnormal{implies} \quad
\lim_{n\rightarrow \infty}\P_{G_n\sim \mathcal{G}_{\textsc{er}}(n,\frac{d}{n}) } (G_n \in \mathcal{A}_n) 
=0. $$
 (see, e.g., \cite{k06}, Theorem 1.1, or \cite{jansonrandomgraphs} for a detailed introduction.) Thus, for $\mathcal{G}_{\textsc{er}}(n,\frac{d}{n})$, we have an analogue of Theorem \ref{thm:thm1}  for $d>d_{\theta+2}$.

\begin{cor}\label{cor:er}
	 Let $\theta\geq 2$ be an integer and $d>d_{\theta+2}$ for $d_{\theta+2}$ defined as (\ref{eq:rcore ER}). \textsf{Whp} over the choice of $G_n \sim \mathcal{G}_{\textsc{er}}(n, \frac{d}{n})$, the discrete-time threshold-$\theta$ contact process $(X_t)$ on $G_n$ satisfies the following:
	 
	 \begin{enumerate}
	 	\item[{\textnormal 1.}] There exist $p_1 \in (0,1)$ and $\varepsilon_1,\gamma_1>0$ such that if $p\geq p_1$ and $X_0\equiv \one$, then $|X_t|\geq (1-\varepsilon_1)\rho_{\theta+2}(\hat{h}) n$ for all $t\in[0, e^{\gamma_1 n}]$ \textsf{whp}, where $\rho_{\theta+2}$ and $\hat{h}$ are as in Proposition \ref{prop:kcore}.
	 	
	 	\item [{\textnormal 2.}] For any $p\in(0,1)$, there exist $\varepsilon_2(p), C_2(p)>0$ such that if $|X_0|\leq \varepsilon_2 n$, then $X_{C_2 \log n} \equiv \textbf{0}$ \textsf{whp}. 
	 \end{enumerate}
\end{cor}

\subsection{Discussons and further problems}

\subsubsection{Emergence of metastability}

%It turns out that our method requires tails of $\mu$ thinner than a power law with a large enough exponent, as described in Theorem \ref{thm:thm4}. A natural question is to establish an analogous result under weaker assumptions on $\mu$, which might be false for some $\mu$ with a fat tail. This leads to a more fundamental question on $p_c$, the metastability emergence threshold (or the long survival threshold).

%\vspace{2mm}
%\noindent \textbf{Question 1.} Is there a degree distribution $\mu$ bounded below by $\theta+2$ such that $p_c(\mu)=0$ or $p_c(\mu)=1$ on $\mathcal{G}(n,\mu)$? Is the phase transition discontinuous for general $\mu$?

%\vspace{2mm}
It turns out that the existence of $(\theta+2)$-cores inside random graphs is not a  necessary condition for metastable states to emerge. For instance, consider $\mu= {\eta} \delta_{3} + (1-\eta) \delta_4 $, where $\delta_a$ is a Dirac point measure at $a$. Then, clearly $G_n\sim\mathcal{G}(n,\mu)$ has an empty $4$-core. However, based on the same method used to prove Theorem \ref{thm:thm1}-1, one can see that the threshold-$2$ contact process on $G_n$ can display an exponentially long survival on $G_n$, for small enough $\eta$ and large enough $p$.

\vspace{2mm}
\noindent \textbf{Question 1.} Find a necessary (and sufficient) condition on $\mu$ for the threshold-$\theta$ contact process on $G_n\sim \mathcal{G}(n,\mu)$ to exhibit an exponentially long survival for large enough $p$.

\subsubsection{Discontinuous phase transition on infinite Galton-Watson trees}

One may be interested in studying the threshold-$\theta$ contact process on infinite Galton-Watson trees, which are local weak limits of random graphs (for details, see, e.g., \cite{DemboMontanari2010}, Section 2.1). However, there does not seem to be an obvious way to translate the results on random graphs (Theorems \ref{thm:thm1}, \ref{thm:thm4}) to those on infinite Galton-Watson trees, which was possible for random $d$-regular graphs and the infinite $d$-regular tree \cite{cd13}. In \cite{cd13}, they used the fact that the local neighborhoods of $\mathcal{G}(n,d)$ look the same as a $d$-regular tree for $n-o(n)$ vertices. Nevertheless, in $\mathcal{G}(n,\mu)$, only a very small fraction of vertices has a local neighborhood structure that looks like a \textit{fixed} Galton-Watson tree. 

\vspace{2mm}
\noindent \textbf{Question 2.} Let $\mu$ be an unbounded degree distribution that has finite $k$-th moments for all $k>0$ and is bounded by $\theta+1$ from below. Show  that for almost every instance $T$ of the infinite Galton-Watson tree with offspring distribution $\mu$, the threshold-$\theta$ contact process on $T$ exhibits a discontinuous phase transition at $p_c\in(0,1)$.
\vspace{2mm}

\noindent Note that if $\mu$ is bounded, say, by $M$, we know the answer by comparing the Galton-Watson tree with the infinite $M$-regular tree \cite{cd13}.

\subsubsection{Extinction on the infinite $(\theta+1)$-regular tree}
Although we proved in Theorem \ref{thm:thm2} that the threshold-$\theta$ contact process on $G_n\sim \mathcal{G}(n,\theta+1)$ does not have a metastable state, it is not clear if the process on the infinite $(\theta +1)$-regular tree always dies out. The proof for Theorem \ref{thm:thm2} cannot be applied, since there are no cycles in the infinite tree, which played a huge role in $\mathcal{G}(n,\theta+1)$.

\vspace{2mm}
\noindent \textbf{Question 3.} Does the threshold-$\theta$ contact process on the infinite $(\theta+1)$-regular tree with all-infected initial condition die out in finite time?

\subsection{Main techniques}\label{subsec:main techniques}
We begin with explaining the main ideas in the proof  of Theorem \ref{thm:thm1}. Working with the generalized setting (Theorem \ref{thm:thm4}) is based on the same strategy, but requires more technicality. 

Let $W$ be the set of infected sites in $G_n\sim \mathcal{G}(n,\mu)$ at some time $t$. Then, to understand the infections at time $t+1$, the information on the following set is  crucial:
\begin{equation}\label{eq:W* def}
W^{*\theta} := \{u: u \textnormal{ has at least }\theta \textnormal{ neighbors in } W\}.
\end{equation}
If we know the size of $W^{*\theta}$, then $|X_{t+1}| \approx p|W^{*\theta}|$.

In order to establish long survival results, we introduce an event $\mathcal{F}^{*l} (m_1, m_2) $ such that
\begin{equation}\label{eq:F*def}
\mathcal{F}^{*l}(m_1,m_2):= 	\{ \textnormal{for all } W \subset V \textnormal{ of size } m_1, ~|W^{*l}|\geq m_2 \},
\end{equation}
where $V$ denotes the set of vertices of $G_n$. One can find a similar definition in \cite{cd13}. Then the key property in our argument can be stated as follows.

\begin{prop}\label{prop:3gen}
	Let $\mu$ be a distribution on $\mathbb{N}$ that satisfies $\E_{D\sim \mu} D^6 = M_6<\infty$ and is lower bounded by $\theta+2$. Then, there exists a constant $\varepsilon_1(\theta,M_6)>0$ such that for $G_n \sim \mathcal{G}(n,\mu)$,
	\begin{equation}\label{eq:prop3}
	\P \left(G_n \in \mathcal{F}^{*\theta} \left( \lceil (1-\varepsilon_1) n\rceil, \lceil(1-0.9\varepsilon_1)n\rceil \right)\right) 
	\geq
	 1- \frac{2\varepsilon_1^{1/5}}{(\varepsilon_1 n)^5 },
	\end{equation}
	for all large enough $n$.
\end{prop}

For the purpose of proving short survival, it suffices to consider the case $\theta=2$. Define $\mathcal{E}^{*2}(m_1,m_2)$ to be an event that
\begin{equation}\label{eq:E*def}
\mathcal{E}^{*2}(m_1,m_2):= 	\{ \textnormal{for all } W \subset V \textnormal{ of size } m_1, ~|W^{*2}|\leq m_2 \}. 
\end{equation}

\begin{prop}\label{prop:2}
	Let $\mu$ be a distribution on $\mathbb{N}$ that  has finite $k$-th moments for all $k>0$. Then, for all $\delta>0$, there exists $\varepsilon_2(\delta,\mu)>0$ such that for $G_n \sim \mathcal{G}(n,\mu)$ and for any $\varepsilon(n)\in (0,\varepsilon_2)$ such that  $\varepsilon n$ is an integer and $\lim_{n\rightarrow \infty} \varepsilon n = \infty$,
	\begin{equation}\label{eq:prop2}
	\P \left(G_n \in \mathcal{E}^{*2} (\varepsilon n, (1+\delta)\varepsilon n) \right) \geq 1-  \frac{2\varepsilon^{1/4} }{(\varepsilon n)^{20/\delta}} - e^{-c \delta \varepsilon n },
	\end{equation}
	for all large enough $n$, where $c>0$ is an  absolute constant.
\end{prop}

%Assuming that this proposition is true, one may then conclude that for $|X_t| \leq \varepsilon n$, $|X_{t+1}| \leq p(1+2\delta)\varepsilon n$ with $p(1+2\delta)$ smaller than $1$, by setting $\delta$  small enough. Note that the probability of having $|X_{t+1}| \geq p(1+2\delta)\varepsilon n$ is exponentially small by a standard large deviation estimate for binomials. 

%To establish Proposition \ref{prop:2}, we work with a fixed subset $W\subset V$ of size $\varepsilon n$ and later apply a union bound for all possible choices of $W$. To study the probability of $|W^{*2}|\geq (1+\delta) \varepsilon n$, we split the set $W^{*2}$ into two parts, $W^{*2}_{in} = W^{*2} \cap W$ and $W^{*2}_{out} = W^{*2} \cap W^c$, and study the event 
%\begin{equation}\label{eq:event split}
%\{|W^{*2}_{in}| \geq \beta \varepsilon n  \} \cap \{|W^{*2}_{out}|\geq (1+\delta -\beta) \varepsilon n \}.
%\end{equation}

%A key observation here is that the two events in (\ref{eq:event split}) are conditionally independent given the number of edges between $W$ and $W^c$. Thus, we can study the events separately. Then, we reduce the problem to an estimation of large deviation events of certain binomials. In particular, we rely on the ``cut-off line algorithm'' by Kim \cite{k06} to investigate the first event in (\ref{eq:event split}). Proof of Proposition \ref{prop:3} is based on the same approach.

%\begin{remark}
	Chatterjee and Durrett \cite{cd13} also derived similar estimates for random regular graphs as the above propositions.  Thanks to regularity, they had better  error probabilies than above, which was roughly $\exp(-c\varepsilon n\log(1/\varepsilon) )$. In particular, the errors were $o(1)$ even for $\varepsilon n = O(1)$.  In our case, existence of high degree vertices prevents us from achieving the error  as small as in \cite{cd13}, especially when $\varepsilon$ is very small.
%\end{remark}

It turns out that Theorem \ref{thm:thm1}-1 follows in a  straight-forward way from Proposition \ref{prop:3gen}, based on a similar argument as in Sections 2 of \cite{cd13}. In showing Theorem \ref{thm:thm1}-2,  one nees to be more careful after obtaining a small size of infections, since the bound (\ref{eq:prop2}) works for $\ep$ such that $\lim_{n\rightarrow \infty} \varepsilon n = \infty$. To deal with such difficulty, we \textit{bound} the changes of infection size by a certain biased random walk which is easier to deal with.
For the proof of  Theorem \ref{thm:thm4}, we require a generalized version of Proposition \ref{prop:2} which we introduce in Section \ref{sec:structure}. 

In \cite{cd13}, they  had the conclusions (\ref{eq:prop3}), (\ref{eq:prop2}) as byproducts of an analysis on $W^{*1}$, but this approach worked only for random regular graphs.
To overcome this difficulty, we  rather study $W^{*2}$ and $W^{*\theta}$ directly, via describing the random matching among half-edges by certain  binomial type  random variables. We carry out this step using the ``cut-off line algorithm'' \cite{k06},  appropriately modified to fit with our setting. In particular, since high-degree vertices  cause problems in applying the algorithm in a way we want, we introduce a method of proving the propositions after getting rid of those problematic vertices.

\subsection{Definition of the random graph $\mathcal{G}(n,\mu)$}\label{subsec:prelim}

Throughout the paper, the random graph $G_n \sim \mathcal{G}(n,\mu)$ is defined in terms of the \textit{configuration model} which is generated as follows.

\begin{itemize}
	\item Let $D_1,\ldots , D_n$ be $n$ i.i.d. samples from $\mu$ conditioned on $\{\sum_{i=1}^n D_i \textnormal{ is even} \}$.
		
	\item Pair all the half-edges uniformly at random.
\end{itemize}
If $\E_{D\sim \mu} D^2 <\infty$, this model is also contiguous with the random \textit{simple} graph chosen uniformly at random among all graphs with degree sequence $\{D_i\}$ (see, e.g., \cite{j09, vanderhofstad17}). Note that the second moment condition $\E D^2 <\infty$ falls into our assumptions in Theorems \ref{thm:thm1} and \ref{thm:thm4}.

\subsection{Organization}

The rest of the paper is organized as follows. In Section \ref{sec:structure}, we discuss the structural properties of $\mathcal{G}(n,\mu)$ and  prove  Propositions \ref{prop:3gen} and \ref{prop:2}. Then, we settle Theorems \ref{thm:thm1}, \ref{thm:thm4} and Corollary \ref{cor:er} in Section \ref{sec:thm}. Finally, the last section is devoted to the proof of Theorem \ref{thm:thm2}.

\section{Structural analysis of random graphs}\label{sec:structure}

The purpose of this section is to establish  Propositions \ref{prop:3gen} and a generalized version of Proposition \ref{prop:2}, which will be used in Section \ref{sec:thm} to prove Theorems \ref{thm:thm1} and \ref{thm:thm4}. Recalling the definitoins of event $\mathcal{E}^{*2}(m_1, m_2)$ in (\ref{eq:E*def}), the generalization of Proposition \ref{prop:2} can be stated in the following way. 

%\begin{proposition}\label{prop:3gen}
%	Let $\theta\geq 2 $ be an integer. For any degree distribution $\mu$ that is lower bounded by $\theta+2$ and satisfies $\E_{D\sim \mu} D^{6} = M_6<\infty$, there exists a constant $\varepsilon_1(\theta, M_6) $ such that for $G_n \sim \mathcal{G}(n,\mu)$
%	\begin{equation}\label{eq:W3}
%	\P \left(G_n \in \mathcal{F}^{*\theta} ((1-\varepsilon_1)n, (1-0.9\varepsilon_1)n \right) \geq 1- 2e^{-\frac{1}{10} \varepsilon_1 n \log(1/\varepsilon_1) },
%	\end{equation}
%	for all large enough $n$.
%\end{proposition}

\begin{proposition}\label{prop:2gen}
	Let $\delta \in (0,1)$ be any given number. For any degree distribution $\mu$ such that $\E_{D\sim \mu } D^{k_0} <\infty$ with $k_0 = 20/\delta$, there exists a constant $\varepsilon_2 (\delta, \mu)$ such that for $G_n \sim \mathcal{G}(n,\mu)$  and for any $\varepsilon(n) \in (0,\varepsilon_2)$ such that $\varepsilon n$ is an integer and $\lim_{n\rightarrow \infty}\varepsilon n =\infty$,
	\begin{equation*}
	\P \left(G_n \in \mathcal{E}^{*2} ( \varepsilon n, (1+\delta) \varepsilon n) \right) \geq 1- \frac{C \varepsilon^{1/4}}{(\varepsilon n)^{k_0 -1}} -e^{-c\delta\varepsilon n},
	\end{equation*}
	for all large enough $n$, where $c,C>0$ are some absolute constants.
\end{proposition}

In the following subsection, we introduce the ``cut-off line algorithm'' \cite{k06}. This plays a crucial role in the proof of the above propositions, which is done in Sections \ref{subsec:2gen} and \ref{subsec:3gen}. 

\subsection{The cut-off line algorithm}\label{subsec:cutoffline}
The cut-off line algorithm \cite{k06} is a simple tool that is very useful in demonstrating structural properties of random graphs. Kim \cite{k06} used this method to give a sharp estimate on the $k$-core threshold of Erd\H{o}s-R\'enyi random graphs. It turns out that the method can also be very useful in the study of $\mathcal{G}(n,\mu)$ as shown in \cite{bnns}. In this subsection, we introduce the algorithm and derive some properties that are used in Sections \ref{subsec:2gen} and \ref{subsec:3gen}.

Suppose that the degree sequence $\{ D_i\}_{i=1}^n$ of the graph $G_n \sim \mathcal{G}(n,\mu)$ is given by i.i.d. samples from $\mu$. Let $V=[n]:=\{1, \ldots , n\}$ be the vertex set of $G_n$ so that $\deg(i)=D_i$. Let $d = \frac{1}{n} \sum_{i=1}^n D_i$ and $d_0 = \E_{D \sim \mu} D$.

Define $\{a_j \}_{j=1}^{dn}$ be the collection of half-edges. 
For each half-edge $a_j$, let $v_j $ be the vertex that $a_j$ is attached to. In other words, the half-edges are associated with the vertices by a map $v: [{dn}] \rightarrow [n]$  satisfying  $|\{j \in[dn]: v_j=i \} |=D_i $ for each $i\in[n]$. To generate the full graph $G_n$, we have to perform a uniform random perfect matching among all half-edges $\{a_j\}$. %In this subsection, we will particularly be interested in the matching of half-edges inside $W=[ m]\subset V$.

In the cut-off line algorithm, we consider the rectangle $R= [ n ] \times [0,1]$. To be specific, on the $xy$-plane, vertices are located at points $1,\ldots, n$ on the $x$-axis, and each vertex is assigned with a vertical interval of length $1$ from $y=0$ to $y=1$. Then, we assign $B_j \sim \,\textnormal{i.i.d.}$ Unif$[0,1]$ to each half-edge $a_j$, and mark the points
$$(v_j, B_j), \quad j=1,\ldots, dn $$
on the rectangle $R$. Note that since $\{B_j\}$ is generated after the map $v:[dn]\rightarrow [n]$ is chosen, $\{B_j\}$ is independent of $v$.

 Suppose that we are attempting to match $a_1$, the first half-edge, to its counterpart chosen uniformly at random. One way to do this is to find $l_1$ such that $B_{l_1} = \max \{B_j : j\in[dn]\setminus\{1\} \}$ and pair $a_1$ with $a_{l_1}$. Since we found the matching pair of $a_1$, we can both remove $a_1$ and $a_{l_1}$ among the collection of half-edges, and repeat the process for $a_2$ (unless $l_1=2$) to find its pair. This procedure defines the \textit{cut-off line algorithm}, which we rigorously state as follows.

\begin{definition}[The cut-off line algorithm]\label{def:cutoffline}
	 Let $R=[n]\times [0,1], \,a_j,\,v_j, $ and $B_j$ as above and set $H_0 = 1$. At the beginning, $dn$ points $(v_j, B_j)$, $j\in [dn]$ are present on $R$. At $t$-th step of the algorithm, we perform the following.
	\begin{enumerate}
		\item Let $J_t$ be the collection of indices $j\in [dn]$ such that the point $(v_{j}, B_{j})$ is present in $R$.
		
		\item Let $j_t$ be the minimal element in $J_t$, and let $l_t = \arg \max \{B_j : j\in J_t\setminus \{j_t \} \}$. 
		
		\item Set $H_t = B_{l_t}$, remove the points $(v_{j_t}, B_{j_t})$, $(v_{l_t}, B_{l_t})$ from $R$, and place an edge between $v_{j_t}$ and $v_{l_t}$ by pairing $a_{j_t}$ and $a_{l_t}$.
	\end{enumerate}
	The horizontal line $y= H_t$ is called the \textit{cut-off line} at step $t$. %Moreover, we can define the same procedure under any reordering of the clones, if the reordering is independent of $\{B_j\}_{j=1}^{dn}$.
\end{definition}

Suppose that we want to match all the half-edges attached at $W  \subset V$ of size $m$. To this end, we first choose the map $v$ in such a way that  $v_j \in W$ for all $j\leq \sum_{i\in W} D_i$, so that the first $\sum_{i\in W} D_i$ half-edges are all at $W$. Then, we perform the cut-off line algorithm until all the points $(v_j, B_j),\; j\leq \sum_{i\in W} D_i$  are removed from $R$. Suppose that after $T$-th step of the algorithm, all half-edges at $W$ are removed for the first time. 

Our first goal is to derive estimates on $H_T$.  For our purpose, it is enough to have some cheap estimates as in \cite{bnns}, nevertheless sharper ones can be found in \cite{k06}. We first make the following simple observation.

\begin{proposition}\label{prop:deterministic clone number bound}
	Under the above setting, we have that
	\begin{equation}\label{eq:deterministic clone number bound}
	|\{j\in [nd]: B_j \geq H_T \}| \leq 2\sum_{i\in W} D_i.
	\end{equation}
\end{proposition}

\begin{proof}
	At each step of the algorithm, two points are removed from $R$. Note that $T\leq \sum_{i\in W} D_i$ by the definition of $T$. Therefore, at most $2\sum_{i\in W} D_i$ points are removed from $R$ after $T$-th step, which upper bound the l.h.s. of (\ref{eq:deterministic clone number bound}).
\end{proof}

From now on, let $m = \varepsilon n$.  Our  focus is on deriving a lower bound on $H_T$, and then we will  later see its connection with the size of $W^{*l}$. Before getting started, we remark two possible cases when $H_T$ can be low. 
\begin{enumerate}
	\item Too many half-edges are attached to $W$, i.e., many vertices in $W$ have very high degree. If so, we need a number of steps of the algorithm to match all of them, resulting in a low $H_T$.
	
	\item The number of half-edges with large $B_j$ is small.
\end{enumerate}

We begin with controlling the first possibility, bounding $\sum_{i\in W} D_i$.  For later purpose, we prove the following lemma that covers all choices of $\varepsilon n$-sized subsets of $[n]$. %However, we postpone the proof until Section \ref{subsec:lem Di}, since it is less related with the main argument of the current section.

\begin{lemma}\label{lem:sum Di}
	 Let $\eta>0$ be an arbitrary number. Suppose that $\E_{D\sim \mu} D^k =M<\infty$ for some $k>1+\eta$ and let $\varepsilon = \varepsilon(n)\in (0,1)$ be  such that $\varepsilon n>0$ is an  integer. Then, we have
	\begin{equation}\label{eq:sum Di}
	\P \left(\textnormal{for all }W\subset V \textnormal{ with }|W|=\varepsilon n,~ \sum_{i\in W}  D_i \leq 3 M^{\frac{1}{k}} \varepsilon^{1- \frac{1+\eta}{k}} n\right) \leq \frac{2\varepsilon^\eta}{(\varepsilon n)^{k-1}},
	\end{equation} 
	for all large enough $n$.
\end{lemma}

The proof of the lemma is based on applying H\"older's inequality and large deviation estimates. Even though the logic is fairly straight-forward, it requires some care on technical details. We suggest the reader to skip the proof if uninterested in details, which would not harm in  understanding the rest of the paper.

\begin{proof}[Proof of Lemma \ref{lem:sum Di}]

	First, note that it suffices to show the inequality for the largest $\varepsilon n$ values among $\{D_i\}_{i=1}^n$. Let $\Delta$ denote the $\varepsilon n$-th largest value among $\{D_i\}_{i=1}^n$, and let $s$ be the maximal integer that satisfies
	\begin{equation}\label{eq:D geq s}
	\P_{D\sim \mu}(D\geq s ) \geq \varepsilon^{1-\frac{\eta}{k}}.
	\end{equation}
	Then, we have that
	\begin{equation*}
	\P ( \Delta < s) \leq \P \left(\Bin(n, \varepsilon^{1-\frac{\eta}{k}}) \leq \varepsilon n\right) = \exp\left(-c_0 n\varepsilon^{1-\frac{\eta}{k}} \right),
	\end{equation*}
	for some absolute constant $c_0>0$, where we obtain the last equality from the fact $\varepsilon n \geq 1$. Let $W$ be the set of $\varepsilon n$ indices of the $\varepsilon n$ lagest values among $\{D_i\}_{i=1}^n$. Then, we claim that
	\begin{equation}\label{eq:sum Di leq ED}
	\P\left(\sum_{i\in W}D_i \leq 2M^{\frac{1}{k}} \varepsilon^{1-\frac{1+\eta}{k}} n + s\varepsilon n \right) \leq \frac{2\varepsilon^\eta}{(\varepsilon n)^{k-1}}.
	\end{equation}
	
	To establish the above inequality, note that given $\{\Delta \geq s \}$, we have
	\begin{equation}\label{eq:sum Di leq D}
	\sum_{i\in W} D_i \leq \sum_{i=1}^n D_i \one_{\{D_i \geq s+1\}} + s\varepsilon n.  
	\end{equation}
	Moreover, H\"{o}lder's inequality and (\ref{eq:D geq s}) implies that
	\begin{equation}\label{eq:D geq s 1}
	\E_{D\sim \mu}D\one_{\{D\geq s+1\} } \leq M^{\frac{1}{k}} \varepsilon^{(1-\frac{\eta}{k})(1-\frac{1}{k})} \leq M^{\frac{1}{k}} \varepsilon^{1-\frac{1+\eta}{k}}.
	\end{equation}
	
	If $Z_i, \; i\in [n]$ are i.i.d.$~$random variables whose c.d.f. satisfies $P(Z_i \geq t) \leq t^{-\alpha}$ for $\alpha>2$, then the large deviation inequalities for heavy-tailed random variables (see e.g., \cite{n79}, \cite{mn98}) tell us that
	\begin{equation*}
	\P\left( \sum_{i=1}^n(Z_i-\E Z_i) \geq bn \right) \leq 2n \P(Z_1 \geq bn),
	\end{equation*}
	for any constant $b>0$ and for all large enough $n$. Applying this to the random variables $D_i \one_{\{D_i \geq s+1 \}}$, we obtain by (\ref{eq:D geq s 1}) that	
	\begin{equation*}
	%	\begin{split}
	\P \left(\sum_{i=1}^n D_i \one_{\{D_i \geq s+1\}} \geq 2M^{\frac{1}{k}} \varepsilon^{1-\frac{1+\eta}{k}} n   \right) \leq 2n \P(D_1 \geq  M^{\frac{1}{k}} \varepsilon^{1-\frac{1+\eta}{k}} n ) 
	\leq \frac{2\varepsilon^\eta}{(\varepsilon n )^{k-1}}.
	%	\end{split}
	\end{equation*}
	Thus, we establish (\ref{eq:sum Di leq ED}).

	Moreover, by Markov's inequality,
	\begin{equation}\label{eq:D geq s 2}
	\P_{D\sim \mu}( D\geq s) \leq \frac{M}{s^k}, \quad \textnormal{hence} \quad s\leq M^{\frac{1}{k}}\varepsilon^{-\frac{1}{k}+\frac{\eta}{k^2}}\leq M^{\frac{1}{k}}\varepsilon^{-\frac{1}{k}}.
	\end{equation}
	Combining (\ref{eq:sum Di leq ED}), (\ref{eq:D geq s 1}) and (\ref{eq:D geq s 2}), we deduce the conclusion.
\end{proof}

Next objective is to study $e(i,W)$, the number of edges between a vertex $i$ and the set $W$. Suppose that $y=H_T$ is the cut-off line when all half-edges from $W$ are matched for the first time. Then, observe that
\begin{equation}\label{eq:e(i,W) coupling}
e(i,W) = \textnormal{the number of points removed from }\{i\}\times [H_T , 1 ] \textnormal{ until }T\textnormal{-th step.}
\end{equation}

Let $h\in[0,1]$ and consider $Y_i \sim \Bin(D_i, 1-h),\;i\in[n]$, where $Y_i$'s are mutually independent. Since the height of the points in $R$ are set to be i.i.d.$\,$Unif$[0,1]$, we can see that given the event $H_T \geq h$, there is a coupling between $\{e(i,W) \}_{i=1}^n$ and $\{Y_i\}_{i=1}^n$ such that 
\begin{equation}\label{eq:e(i,W) leq Yi}
e(i,W) \leq Y_i \textnormal{ for all }i\in[n], \textnormal{ conditioned on }\sum_{i=1}^n Y_i \geq 
2\sum_{i'\in W} D_{i'},
%6M^{\frac{1}{k}}{\varepsilon^{1-\frac{1+\eta}{k }}}n,
\end{equation}
where the latter conditioning is needed to set the cut-off line above $h$, based on Proposition \ref{prop:deterministic clone number bound}. Define the graph property $\mathcal{B}_n$ (i.e., subsets of graphs of $n$ vertices) by
\begin{equation}\label{eq:Bn def}
\mathcal{B}_n := 	\left\{\textnormal{for all }W\subset V \textnormal{ with } |W|=\varepsilon n,~  \sum_{i\in W}  D_i \leq 3 M^{\frac{\delta}{20}} \varepsilon^{1- \frac{\delta}{16}} n \right\},
\end{equation}
the event given in (\ref{eq:sum Di}). Now we choose appropriate $h$, which can be done based on the following estimate on $H_T$.

\begin{lemma}\label{lem:HT bound}
	Let $\eta>0$ be arbitrary. Suppose that $\E_{D\sim \mu} D^k =M<\infty$ for some $k>1+\eta$, $\E_{D\sim\mu}D =d_0$ and let $\varepsilon\in (0,1)$ be a number such that $\varepsilon n$ is an integer. Let $W\subset V$ be any subset satisfying $|W|=\varepsilon n$, and let $H_T$ be the height of the cut-off line when all clones at $W$ are removed from the rectangle $R$ for the first time. Then,  we have that for all large enough $n$,
	\begin{equation}\label{eq:cutoff height bound}
	\P \left(\left. 1- H_T \geq \frac{7 M^{\frac{1}{k}}\varepsilon^{1-\frac{1+\eta}{k}} }{d_0} \;\right|\;\mathcal{B}_n \right) \leq \exp\left\{ -c \varepsilon^{1-\frac{1}{k}} n \right\},
	\end{equation}
	where $\mathcal{B}_n$ is given as (\ref{eq:Bn def}) and $c>0$ is an absolute constant.
\end{lemma}

\begin{proof}
	Recall that $d=\frac{1}{n} \sum_{i=1}^n D_i$. We first note that if $H_T \leq h$, then 
	\begin{equation*}
	|\{j \in[dn] : B_j \geq h \}| \leq 2\sum_{i\in W} D_i,
	\end{equation*}
	by Proposition \ref{prop:deterministic clone number bound}. Moreover, observe that given $\{D_i\}_{i=1}^n$, 
	\begin{equation*}
	|\{j \in [dn] : B_j \geq h \}| \overset{\textnormal{d}}{=} \Bin(dn, 1-h).
	\end{equation*}
	We also know that \textsf{whp} over the choice of $\{D_i\}_{i=1}^n$, $\sum_{i\in W} D_i \leq 3M^{\frac{1}{k}}\varepsilon^{1-\frac{1+\eta}{k}} n$ for all $W$. Further, we have $d\geq 0.9 d_0$ \textsf{whp}. Therefore, we obtain that the l.h.s. of (\ref{eq:cutoff height bound}) is at most
	\begin{equation}\label{eq:HT by binom estim}
	\P \left(\Bin\left(0.9d_0n, \frac{7 M^{\frac{1}{k}}\varepsilon^{1-\frac{1+\eta}{k}} }{d_0} \right)    \leq 6M^{\frac{1}{k}}\varepsilon^{1-\frac{1+\eta}{k}} n \right).
	\end{equation}
	To control this term, we rely on the following large deviation estimate for binomials: for all $\gamma>0$, there exists $c_\gamma>0$ such that
	\begin{equation}\label{eq:large dev bin 1}
	\P\left(\Bin(n,p) \leq (1-\gamma)np \right) \leq \exp(-c_\gamma np).
	\end{equation}
	(See, e.g., Chapter 21 of \cite{fk16}.) Applying this to (\ref{eq:HT by binom estim}), we see that  it is bounded by  $\exp\left\{ -c \varepsilon^{1-\frac{1}{k}} n \right\}$,	
	for some absolute constant $c>0$, and hence we obtain the conclusion.
\end{proof}

Therefore, the lemma tells us that we can choose $h$ to be
\begin{equation*}
h= 1-\frac{7 M^{\frac{1}{k}}\varepsilon^{1-\frac{1+\eta}{k}} }{d_0},
\end{equation*}
so that   conditioned on $\mathcal{B}_n$,
\begin{equation}\label{eq:sum Yi bound}
\P \left(\left.\sum_{i=1}^n Y_i \geq 6M^{\frac{1}{k}}{\varepsilon^{1-\frac{1+\eta}{k }}}n \;\right|\;\mathcal{B}_n \right) \geq 1-\exp\left\{-c \varepsilon^{1-\frac{1}{k}} n \right\}.
\end{equation}

\subsection{Proof of Proposition \ref{prop:2gen}}\label{subsec:2gen}

Let $\delta\in (0,1)$ be given and suppose that $\mu$ satisfies $\E_{D\sim \mu} D^{k_0} =M<\infty$ with $k_0 = \frac{20}{\delta}$, as in the statement of Proposition \ref{prop:2gen}. Let the degree sequence $\{D_i\}_{i=1}^n$ be given by i.i.d. samples of $\mu$ and let $\varepsilon=\varepsilon(n)\in(0,1)$ be a small parameter which will be specified later.

Before considering the matching between half-edges, we first exclude vertices with too large $D_i$ from our consideration by a simple concentration argument.

\begin{lemma}\label{lem:large Di}
	Suppose that $\mu$ satisfies $\E_{D\sim \mu} D^{k_0} =M<\infty$ with $k_0 = \frac{20}{\delta}$ and $\varepsilon=\varepsilon(n)\in(0,1)$ obeys $\lim_{n\rightarrow \infty}\varepsilon n=\infty$. Then, there exists an absolute constant $c>0$ such that for all large enough $n$ and $D_i \sim \textnormal{i.i.d.}~\mu$,
	\begin{equation}\label{eq:number of large Di}
	\P_\mu \left(\left|\left\{i\in[n]: D_i \geq \left(\frac{3M}{\delta\varepsilon} \right)^{\delta/20} \right\}\right|  \geq \frac{1}{2}\delta\varepsilon n\right) \leq \exp \{-c\delta\varepsilon n \} .
	\end{equation}
\end{lemma}

\begin{proof}
	Let $M_0 = (3M/\delta \varepsilon)^{\delta/20}$. Then, note that
	\begin{equation*}
	|\{i\in[n]: D_i \geq M_0 \}| \sim \Bin \left(n, \P_{\mu}(D\geq M_0)\right). 
	\end{equation*}
	We can easily see that 
	$$\P_\mu (D\geq M_0) \leq \frac{M}{M_0^{20/\delta}} = \frac{1}{3} \delta \varepsilon,$$
	and hence (\ref{eq:number of large Di}) follows from (\ref{eq:large dev bin 1}).
\end{proof}

Let  $\mathcal{A}_n$ be the graph property  defined as
\begin{equation}\label{eq: An Bn def}
\begin{split}
\mathcal{A}_n := \left\{\left|\left\{i\in[n]: D_i \geq \left(\frac{3M}{\delta\varepsilon} \right)^{\delta/20} \right\}\right|  < \frac{1}{2}\delta\varepsilon n\right\},
\end{split}
\end{equation}
where $\{D_i\}_{i=1}^n$ is the degree sequence of a graph, and let $\mathcal{B}_n$ be as (\ref{eq:Bn def}). Then for $G_n \sim \mathcal{G}(n,\mu)$, we obtain the following by applying Lemmas \ref{lem:sum Di} and \ref{lem:large Di} with $\eta = \frac{1}{4}$.
\begin{equation}\label{eq:An cap Bn}
\P \left(G_n \in \mathcal{A}_n \cap \mathcal{B}_n \right) \geq 1- \frac{2\varepsilon^{1/4}}{(\varepsilon n)^{k_0 -1}} -\exp(-c\delta \varepsilon n ),
\end{equation}
for some absolute constants $c, C>0$. Also, note that the event $\mathcal{A}_n \cap \mathcal{B}_n$ is measureable with respect to $\{D_i\}_{i=1}^n$, i.e., independent of the matching between half-edges.

Let $W\in V=[n]$ be a fixed subset of vertices with size $|W|=\varepsilon n$. As discussed in the previous subsection, suppose that we generate $G_n$, by first pairing the half-edges at $W$ following the cut-off line algorithm. Let $H_T$ be the cut-off line when all half-edges at $W$ are paired for the first time. 

Let $d_0 = \E_{D\sim\mu}D$.  By applying Lemma \ref{lem:HT bound} with $\eta=\frac{1}{4}$, we have
\begin{equation}\label{eq:cutoff height bound for 2}
\P \left(1- H_T \geq \frac{7 M^{\frac{\delta}{20}}\varepsilon^{1-\frac{\delta}{16} }}{d_0} \right) \leq \exp\left\{ -c \varepsilon^{1-\frac{\delta}{20}} n \right\},
\end{equation}
for some absolute constant $c>0$. Let 
$$h = 1- \frac{7 M^{\frac{\delta}{20}}\varepsilon^{1-\frac{\delta}{16} }}{d_0},$$
and for each $i\in[n]$ define
$Y_i \sim \Bin (D_i , 1-h)$ to be mutually independent given $\{D_i\}_{i=1}^n$. Then,
\begin{equation}\label{eq:W2 estim split}
\begin{split}
&\P\left(\left.|W^{*2}| \geq (1+\delta)\varepsilon n \,\right|\,\mathcal{A}_n \cap \mathcal{B}_n \right)\\
&~~~\leq
\P \left(\left. \sum_{i=1}^n Y_i \leq 6M^{\frac{\delta}{20}} \varepsilon^{1-\frac{\delta}{16} }n \,\right|\, \mathcal{A}_n \cap \mathcal{B}_n  \right)\\
&~~~~~+\P\left( \left| \left\{ i\in[n] : Y_i\geq 2 \right\} \right| \geq(1+\delta) \varepsilon n\,   \left|\, \sum_{i=1}^n Y_i \geq 6M^{\frac{\delta}{20}} \varepsilon^{1-\frac{\delta}{16} }n,~\mathcal{A}_n \cap \mathcal{B}_n \right.  \right).
\end{split}
\end{equation}
The first term in the r.h.s. can be handled by (\ref{eq:sum Yi bound}). For the second, we carry out by bounding $Y_i$ by  $Z_i \sim (M_0, 1-h)$ with  $M_0 = (3M/\delta \varepsilon)^{\delta/20}$, where $Z_i$ being independent of each other. Namely,
\begin{equation}\label{eq:W2 estim 2}
\begin{split}
&\P\left( \left| \left\{ i\in[n] : Y_i\geq 2 \right\} \right| \geq(1+\delta) \varepsilon n\,   \left|\, \sum_{i=1}^n Y_i \geq 6M^{\frac{\delta}{20}} \varepsilon^{1-\frac{\delta}{16} }n,~\mathcal{A}_n \cap \mathcal{B}_n \right.  \right) \\
&~~~~~~~~\leq 
(1+o(1)) \P (\left| \left\{ i\in[n] : Y_i\geq 2 \right\} \right| \geq(1+\delta) \varepsilon n \,|\, \mathcal{A}_n)\\
&~~~~~~~~\leq
(1+o(1)) \P \left(\left| \left\{ i\in[n] : Z_i\geq 2 \right\} \right| \geq\left(1+\frac{\delta}{2}\right) \varepsilon n \right),
\end{split}
\end{equation}
 where the last inequality follows from the definition of $\mathcal{A}_n$ and Lemma \ref{lem:large Di}. We can then bound
\begin{equation*}
\P (Z_i \geq 2) \leq M_0^2 (1-h)^2,
\end{equation*}
and hence obtain that $ \left| \left\{ i\in[n] : Z_i\geq 2 \right\} \right| \leq_{\textsc{st}} \Bin (n, M_0^2 (1-h)^2)$.
Note that the binomials satisfy the following large deviation estimate (see, e.g., Chapter 21 of \cite{fk16})
\begin{equation}\label{eq:bin large dev}
\P(\Bin(n,p) \geq \mu n p) \leq \exp(-\mu n p \log(\mu/e)).
\end{equation}
We thus obtain
\begin{equation}\label{eq:W2 estim 3}
\textnormal{r.h.s. of (\ref{eq:W2 estim 2}) }\leq
(1+o(1))\exp \left\{-\left(1+\frac{\delta}{2} \right) \log\left( \varepsilon^{-1+\frac{\delta}{3}} \right)  \right\},
\end{equation}
if $\varepsilon$ satisfies
\begin{equation}\label{eq:varep def}
\varepsilon^{\frac{\delta}{12}} \leq \frac{d_0^2 \delta^{\frac{\delta}{10}}}{150e M^{\frac{\delta}{5}}}.
\end{equation}

Therefore, combining (\ref{eq:sum Yi bound}), (\ref{eq:W2 estim split}) and (\ref{eq:W2 estim 3}) implies that
\begin{equation*}
\P\left(\left.|W^{*2}| \geq (1+\delta)\varepsilon n \,\right|\,\mathcal{A}_n \cap \mathcal{B}_n \right)
\leq
(1+o(1))\exp \left\{-\left(1+\frac{\delta}{6} \right) \varepsilon n \log\left(\frac{1}{\varepsilon} \right)  \right\}.
\end{equation*}
Then, we can apply a union bound to obtain that
\begin{equation}\label{eq:W2 union bd}
\begin{split}
&\P\left(\exists \, W\subset V \textnormal{ such that }|W|=\varepsilon n \textnormal{ and } \left.|W^{*2}| \geq (1+\delta)\varepsilon n \,\right|\,\mathcal{A}_n \cap \mathcal{B}_n \right)\\
&~~~~~~\leq
(1+o(1)) {n \choose \varepsilon n } \exp \left\{-\left(1+\frac{\delta}{6} \right) \varepsilon n \log\left(\frac{1}{\varepsilon} \right)  \right\}\\
&~~~~~~\leq
(1+o(1)) \exp \left\{-\frac{\delta}{6}  \varepsilon n \log\left(\frac{1}{\varepsilon} \right)  \right\}.
\end{split}
\end{equation}
Finally, we deduce the conclusion by simply observing
\begin{equation*}
	\P \left(G_n \in \mathcal{E}^{*2} ( \varepsilon n, (1+\delta) \varepsilon n) \right)
	\geq \P(\mathcal{A}_n \cap \mathcal{B}_n )
	\, \P \left(\left.G_n \in \mathcal{E}^{*2} ( \varepsilon n, (1+\delta) \varepsilon n) \,\right|\,\mathcal{A}_n \cap \mathcal{B}_n \right),
\end{equation*}
and then using (\ref{eq:An cap Bn}), for any $\varepsilon$ satisfying (\ref{eq:varep def}). \qed

\subsection{Proof of Proposition \ref{prop:3gen}}\label{subsec:3gen}

Let $\theta \geq 2$ be an integer, and suppose that $D\sim \mu$ satisfies $\P(D\geq \theta+2)=1$ and $\E D^6= M <\infty$. Also, let $\varepsilon>0$ be such that $\varepsilon n$ is an integer. If we generate the degree sequence of $G_n$ by i.i.d. samples of $\mu$. As  
(\ref{eq:Bn def}), define $\mathcal{B}_n$ to be an event that
\begin{equation*}
\mathcal{B}_n := 	\left\{\textnormal{for all }W\subset V \textnormal{ with } |W|=\varepsilon n,~  \sum_{i\in W}  D_i \leq 3 M^{\frac{1}{k_0}} \varepsilon^{1- \frac{1+\eta}{k_0}} n \right\},
\end{equation*}
where $k_0, \eta>0$ can be any numbers such that $1+\eta < k_0 \leq 6$. Then, Lemma \ref{lem:sum Di}  implies that for some absolute constant $C>0$,
\begin{equation*}
\P \left(G_n \in \mathcal{B}_n \right) \geq 1- \frac{2\varepsilon^\eta}{(\varepsilon n)^{k_0-1}}.
\end{equation*}

Now, let $W\subset V$ be a fixed subset of vertices with size $|W|=\varepsilon n$, and let $U = V\setminus W$. We will study the set $V\setminus U^{*\theta}$, which can be written as
\begin{equation*}
V\setminus U^{*\theta} = \{i\in[n]: e(i,W)\geq D_i -\theta+1    \},
\end{equation*}
where $e(i,W)$ denotes the number of edges between vertex $i$ and the set $W$. As in the previous subsection, consider the independent binomials $Y_i \sim \Bin(D_i,1-h)$, where 
\begin{equation}\label{eq:h def}
h=1-\frac{7M^{\frac{1}{k_0}}\varepsilon^{1-\frac{1+\eta}{k_0} }}{d_0},
\end{equation}
with $d_0=\E_{D\sim \mu}D$. Suppose that we generate $G_n$ starting from matching the half-edges at $W$, and let $H_T$ is the height of the cut-off line when all half-edges at $W$ are paired for the first time. Then, by (\ref{eq:e(i,W) leq Yi}), similar argument as (\ref{eq:W2 estim split}) tells us that
\begin{equation}\label{eq:W3 estim split}
\begin{split}
&\P \left(\left.\left|V\setminus U^{*\theta} \right| \geq 0.9\varepsilon n\,\right|\, \mathcal{B}_n \right)\\
&~~~\leq
\P \left(\left. \sum_{i=1}^n Y_i \leq 6M^{\frac{1}{k_0}} \varepsilon^{1-\frac{1+\eta}{k_0} }n \,\right|\,  \mathcal{B}_n  \right)\\
&~~~~~+\P\left( \left| \left\{ i\in[n] : Y_i\geq D_i-\theta+1 \right\} \right| \geq 0.9 \varepsilon n\,   \left|\, \sum_{i=1}^n Y_i \geq 6M^{\frac{1}{k_0}} \varepsilon^{1-\frac{1+\eta}{k_0} }n,~ \mathcal{B}_n \right.  \right).
\end{split}
\end{equation}
Observe that
\begin{equation*}
\P (\Bin(d, 1-h) \geq d-\theta+1 ) \leq {d \choose \theta -1}  (1-h)^{d-\theta+1} =: \mathfrak{p}_{\theta, h}(d).
\end{equation*}
Further, we can find the maximum of $\mathfrak{p}_{\theta, h}(d)$ in the regime $d \geq \theta +2$ by noting that
\begin{equation*}
\frac{\mathfrak{p}_{\theta,h}(d+1) }{\mathfrak{p}_{\theta,h}(d) }
=\frac{d+1}{d-\theta+2}(1-h) \leq \frac{\theta+3}{4}(1-h).
\end{equation*}
Hence, if $1-h \leq \frac{4}{\theta +3}$, then $\mathfrak{p}_{\theta, h}(d)$ is maximal when $d=\theta+2$. Therefore, for such $h$, we have
\begin{equation}\label{eq:Yi geq Di -theta+1}
\P(Y_i \geq D_i -\theta+1  ) \leq {\theta+2 \choose \theta-1}(1-h)^3 \leq \{(\theta+2)(1-h)\}^3.
\end{equation}

Thus, having (\ref{eq:bin large dev}) and (\ref{eq:h def}) in mind, we see that  the last term in (\ref{eq:W3 estim split}) is bounded by
\begin{equation*}
\begin{split}
(1+o(1)) &\P\left( \left| \left\{ i\in[n] : Y_i\geq D_i-\theta+1 \right\} \right| \geq 0.9 \varepsilon n   \right)\\
&\leq
(1+o(1)) \P \left(\Bin(n, \{(\theta+2)(1-h)\}^3) \geq 0.9\varepsilon n \right)\\
&\leq
\exp\left\{-0.9 \varepsilon n \log \left( \varepsilon^{-2+\frac{4+3\eta }{k_0}} \right) \right\},
\end{split} 
\end{equation*}
if $\varepsilon$ satisfies 
\begin{equation}\label{eq:epsilon bd 2}
\varepsilon \leq e^{-k_0} \left(\frac{d_0}{ 7(\theta+2)M} \right)^{3k_0}.
\end{equation}
Note that we used (\ref{eq:bin large dev}) for the last inequality. Therefore, if we take $k_0 = 6$, $\eta = \frac{1}{5}$, we get
\begin{equation}\label{eq:epsilon bd 3}
\begin{split}
\P \left(\left.\left|V\setminus U^{*\theta} \right| \geq 0.9\varepsilon n\,\right|\, \mathcal{B}_n \right)
&\leq
\exp\left\{-c \varepsilon^{1-\frac{1}{6}}n \right\} + \exp \left\{-1.1\varepsilon n \log\left(\frac{1}{\varepsilon} \right) \right\}\\
&\leq 
2\exp \left\{-1.1\varepsilon n \log\left(\frac{1}{\varepsilon} \right) \right\},
\end{split}
\end{equation}
where $c>0$ is an absolute constant from (\ref{eq:sum Yi bound}), and the last inequality holds true if we take $\varepsilon$ small so that $c\varepsilon^{1/6} \geq \log(1/\varepsilon)$. Therefore, the conclusion (\ref{eq:prop3}) follows after we take a union bound as in (\ref{eq:W2 union bd}), if we set $\varepsilon_1$ to be small so that it satisfies
\begin{equation}\label{eq:epsilon1 def}
\varepsilon_1^{1-\frac{1+\eta}{6}} \leq \frac{4d_0}{7(\theta+2) M^{\frac{1}{6}} }, ~~~~
	\varepsilon_1 \leq e^{-k_0} \left(\frac{d_0}{ 7(\theta+2)M} \right)^{3k_0},~~~ 
	\textnormal{and} ~~
	c\varepsilon_1^{\frac{1}{6}} \geq \log(1/\varepsilon_1),
\end{equation}
where each condition follows from (\ref{eq:Yi geq Di -theta+1}), (\ref{eq:epsilon bd 2}) and (\ref{eq:epsilon bd 3}), respectively.
\qed

\section{Discontinuous phase transitions on random graphs}\label{sec:thm}

In this section, we establish Theorems \ref{thm:thm1} and \ref{thm:thm4}. Theorem \ref{thm:thm1}-1 follows by a simple argument, similarly  as Sections 2 \cite{cd13}. Theorem \ref{thm:thm2} requires additional care as mentioned in Section \ref{subsec:main techniques}, where we discuss a biased random walk argument to control the expansion of infections when their size is small. In order to settle Theorem \ref{thm:thm4}, we introduce some background on the $k$-cores of random graphs in Section \ref{subsec:thm4}.

\subsection{Proof of Theorem \ref{thm:thm1}-1}\label{subsec:thm1-1}

Let $\delta_1 = 0.1 $ and $\varepsilon_1 >0$ be the constant as in Proposition \ref{prop:3gen}. Set
\begin{equation*}
p_1 : = \frac{1-\varepsilon_1}{1-(1-\frac{\delta_1}{2}) \varepsilon_1}.
\end{equation*}
Suppose that the threshold-$\theta$ contact process $(X_t)$ with $p\geq p_1$  satisfies $|X_t|\geq (1- \varepsilon_1) n$ at some time $t$. Recalling the definiton of $\mathcal{F}^{*l}(m_1, m_2)$ in (\ref{eq:E*def}), we have  
\begin{equation*}
G_n \in \mathcal{F}^{*\theta}(\lceil(1-\varepsilon_1) n\rceil, \lceil(1-(1-\delta_1)\varepsilon_1) n\rceil) \quad \textnormal{implies} \quad |X_{t+1}| ~\geq_{\textsc{st}}~ \Bin \left(\lceil (1-(1-\delta_1)\varepsilon_1 ) n\rceil, \, p \right),
\end{equation*}

\noindent where $\geq_{\textsc{st}}$ denotes stochastic domination between random variables. Thus, on the event $\mathcal{F}^{*\theta}((1-\varepsilon_1) n, (1-(1-\delta_1)\varepsilon_1) n)$, a standard large deviation estimate  for binomials (e.g., \cite{fk16}, Section 22.4) implies that
\begin{equation*}
\P_{\textsc{tcp}} \left(|X_{t+1}| \leq (1-\varepsilon_1)n \, | \, |X_t| \geq (1-\varepsilon_1)n \right) \leq \exp \left(-\frac{\delta_1^2 \varepsilon_1^2}{12} n \right),
\end{equation*}
where $\P_{\textsc{tcp}}$ denotes the probability coming from the randomness of $(X_t)$. Set $\tau := \exp\left(\frac{\delta_1^2 \varepsilon_1^2}{13}n\right)$, and let $X_0 \equiv \one$. Then, the probability that $ |X_t|\geq (1-\varepsilon_1)n$ fails for some $t\in[0,\tau]$ is exponentially small in $n$. Finally, Proposition \ref{prop:3gen} tells us that $G_n \in \mathcal{F}^{*\theta}((1-\varepsilon_1) n, (1-(1-\delta_1)\varepsilon_1) n)$  \textsf{whp}. \qed

\subsection{Proof of Theorem \ref{thm:thm1}-2}\label{subsec:thm1-2}
As discussed in Section \ref{subsec:main techniques}, the argument in the previous section or in Section 3 of \cite{cd13} does not work the same for Theorem \ref{thm:thm1}-2. Instead, the proof will proceed in two steps as follows.
\begin{enumerate}
	\item Starting from $\varepsilon_2 n$ infections where $\varepsilon_2>0$ is a small fixed constant, we first reduce the size of infection into $\Delta \log\log n $ in time $C_1 \log n$, where $\Delta>0$ is another small fixed constant. 
	
	\item We eliminate the remaining infections \textit{by luck}, meaning that we rely on the event where all the infections are recovered in a single time step. If this fails and the infection expands into a bigger size than $\Delta \log\log n$, then we show that it quickly returns to size $\Delta \log\log n$ after a short time, where we try  another update to kill every infection at once. 
\end{enumerate}

We prove the short survival result for the threshold-$2$ contact process, which clearly implies the result for general $\theta \geq 2$. Let $p \in (0,1)$ be the given probability parameter of the threshold-$\theta$ contact process $(X_t)$,  and set $$\delta_2:= \frac{1}{3}(1-p), \qquad \Delta = (3\log (\frac{1}{1-p}))^{-1}.$$ Further, let $\varepsilon_2 = \varepsilon_2(\delta_2)$ as in Proposition \ref{prop:2}. Assume that  $G_n \sim \mathcal{G}(n,\mu)$ satisfies $\mathcal{E}^{*2}(m, (1+\delta_2)m)$ for all $m\in [\Delta \log\log n ,\,\varepsilon_2 n]$, which happens with probability
\begin{equation}\label{eq:prop Cn}
\begin{split}
\P\left(G_n \in \bigcap_{m=\Delta \log\log n}^{\varepsilon_2 n} \mathcal{E}^{*2} (m, (1+\delta_2) m ) \right)
& \geq 1- \sum_{m=\Delta\log\log n}^{\varepsilon_2 n} \left\{ \frac{C}{m^{k_2-1}}\left(\frac{m}{n} \right)^{\frac{1}{4}} +e^{-c\delta m} \right\}\\
&\geq
1- 2\left(\log n \right)^{-c\delta \Delta} = 1-o(1),
\end{split} 
\end{equation}
where we set $k_2 = 20/\delta_2$.
% Denote the event in the left by  $\mathcal{C}_n$, i.e.,

\subsubsection{Reducing the infection to a small size}\label{subsec:short surv step1}

Let us define the event $\mathcal{C}_n$ by
\begin{equation*}
\mathcal{C}_n := \bigcap_{m=\Delta \log\log n}^{\varepsilon_2 n} \mathcal{E}^{*2} (m, (1+\delta_2) m ),
\end{equation*}
the event in the l.h.s. of (\ref{eq:prop Cn}).

To begin with, we show that $(X_t)$ started from $|X_0|\leq \varepsilon_2 n$ reaches size $\Delta \log\log n$ in time $O(\log n)$ \textsf{whp}. Therefore, let us assume that $|X_0| \geq \Delta\log \log n$, otherwise there is nothing to prove.

On $G_n \in \mathcal{C}_n$, if $X_t$ at some time $t$ is  $\Delta \log\log n \leq |X_t| \leq \varepsilon_2 n$, we have
$$|X_{t+1}| \leq_{\textsc{st}} \Bin( (1+\delta_2)|X_t|,\, p).$$
Noting that $p(1+\delta_2)\leq 1-2\delta_2$, we obtain $\E_{\textsc{tcp}}[|X_{t+1}| \, | \, |X_t|]\leq (1-2\delta_2) |X_t| $, and
\begin{equation}\label{eq:short surv one step}
\P_{\textsc{tcp}} \left(|X_{t+1} |\geq (1-\delta_2)m \, |\, |X_t|=m,\, \mathcal{C}_n \right) \leq \exp \left(- \frac{\delta_2^2 m}{3} \right),
\end{equation}
 Let $\tau_2'$ be the first time when $|X_t| \leq \sqrt{n}$.
 Then, there exists some constant $C_2'$ such that
\begin{equation}\label{eq:C2prime}
\P_{\textsc{tcp}} \left(\left. \tau_2' \geq C_2' \log n \,\right|\, |X_0| \leq \varepsilon_2 n, \, \mathcal{C}_n \right) \leq \exp \left(-\frac{\delta_2^2 n^{1/2}}{4} \right),
\end{equation}
where the r.h.s. indicates the probability that the events $\{|X_{t+1}| \leq (1-\delta_2) |X_t| \}$ hold for all $t$ until the infection size reaches $\sqrt{n}$.

To make $(X_t)$ reach $\Delta \log\log n$ from $|X_{\tau_2'}| \leq n^{1/2}$, we consider a coupling between $(X_t)$ and a biased random walk which we define as follows. Let $Q_s$, $s\in[n]$ be i.i.d. random variables given by
\begin{equation*}
\P(Q_s = 1+\delta_2) = (\log n)^{- \gamma}, ~~~~\P(Q_s = 1-\delta_2) = 1- (\log n)^{- \gamma},
\end{equation*}
where $\gamma:= \delta_2^2 \Delta/3$. Here, $\P(Q_s = 1+\delta_2) $ corresponds to the minimum of the r.h.s. of (\ref{eq:short surv one step}) over $m\in [\Delta \log\log n ,\, \varepsilon_2 n]$. Define $(\widehat{X}_t)_{t\geq \tau_2'} $ by $\widehat{X}_{\tau_2'} = \sqrt{n}$, and
\begin{equation*}
\widehat{X}_{\tau_2' +t} = \widehat{X}_{\tau_2'} \prod_{s=1}^t Q_s.
\end{equation*}
Let $T>0$. Then, conditioned on the events 
\begin{equation}\label{eq:short surv events def}
G_n \in \mathcal{C}_n 
% ~~~ |X_{\tau_2'}| \leq \sqrt{n},
 ~~\textnormal{and} ~~
|X_{\tau_2'+t}| \in [\Delta \log\log n,\,\varepsilon_2 n ] ~\textnormal{ for all } t\in[0, T],
\end{equation}
 we have a natural coupling between $(X_t)$ and $ (\widehat{X}_t)$ such that
 \begin{equation*}
 \left|X_{\tau_2' +t}\right| \leq \widehat{X}_{\tau_2'+t}, ~~\textnormal{for all } t\in[0,T].
 \end{equation*}

Let $t_2$ set to be $$t_2 = \left\lceil \frac{3\log n}{2\log (1+\delta_2)} \right\rceil. $$
(Note that $t_2$ is deterministic whereas $\tau_2'$ is not.) Then, we can control the number of $Q_s$ that equals $1+\delta_2$ as follows.
\begin{equation}\label{eq:number of Qs bd}
\begin{split}
\P \left( \left| \left\{s\in[1,t_2] : Q_s = 1+\delta_2 \right\} \right| \geq \frac{t_2}{4} \right) 
= \P \left(\Bin(t_2, \, (\log n)^{-\gamma} )\geq \frac{t_2}{4} \right) =o(1).
\end{split}
\end{equation}
Since $(1-\delta)(1+\delta) <1$, this means that $(\widehat{X}_{\tau_2'+t})$ reaches $\Delta \log\log n$ before $\varepsilon_2 n$ in time $t\leq t_2$ \textsf{whp}. This implies that if $G_n \in \mathcal{C}_n$, the second event in (\ref{eq:short surv events def}) holds true \textsf{whp} until $|X_t|$ reaches $\Delta \log\log n$ at some $t\leq \tau_2' + t_2$. Thus, we obtain that on $G_n \in \mathcal{C}_n$, $|X_t|$ becomes at most $\Delta \log\log n$ in time $C_2'\log n+t_2  = O(\log n)$ \textsf{whp} for $C_2'$ as in (\ref{eq:C2prime}), if started from $|X_0|\leq \varepsilon_2 n$.

\subsubsection{Elemination of small infections}\label{subsec:short surv step2}
Under the same setting as the previous subsection, define $\tau_2 (0)$ to be the first time when $|X_t|$ is at most $\Delta \log\log n$. We showed that there exists a constant $C_2$ such that
\begin{equation*}
\P(\tau_2(0) \geq C_2 \log n \,|\, |X_0|\leq \varepsilon_2 n , \,  \mathcal{C}_n ) = o(1).
\end{equation*}

For each $i\in\mathbb{N}$, define $\tau_2(i)$ inductively so that
\begin{equation*}
\tau_2(i+1) := \min \{ t \geq \tau_2(i)+1 : |X_t| \leq \Delta \log\log n \}.
\end{equation*}
At time $\tau_2(0)+1$, we obtain $X_{\tau_2(0)+1} = \zero$ with probability at least
\begin{equation}\label{eq:small inf extinction prob}
(1-p)^{\Delta \log\log n} = (\log n)^{-1/3}.
\end{equation}
If we fail to achieve such a lucky event, we claim that 
\begin{equation}\label{eq:one step of small infection elimination}
\P\left( \left.\tau_2(1) - \tau_2(0) \geq \sqrt{\log n} \,\,\right|\,\tau_2(0),\,\mathcal{C}_n \right) \leq \exp\left(-\frac{1}{4}\sqrt{\log n} \right), 
\end{equation}
for some $\alpha>0$. This can be shown using a similar estimate as (\ref{eq:number of Qs bd}). That is, 
\begin{equation}\label{eq:small infection elim 1}
\begin{split}
\P &\left( \left| \left\{s\in\left[1, \sqrt{\log n} \right] : Q_s = 1+\delta_2 \right\} \right| \geq \frac{\sqrt{\log n}}{4} \right) \\
&= \P \left(\Bin(\sqrt{\log n}, \, (\log n)^{-\gamma} )\geq \frac{\sqrt{\log n}}{4} \right)
\leq
\exp\left(-\frac{1}{4} \sqrt{\log n} \right)
,
\end{split}
\end{equation}
where the last inequality follows from (\ref{eq:bin large dev}). Noting that $|X_{\tau_2(0)+1}| \leq 2\Delta\log\log n$, the event inside the l.h.s. of (\ref{eq:small infection elim 1}) implies that $(X_{\tau_2(0)+t})$ reaches $\Delta\log\log n$ before $\varepsilon n$. Of course, this happens the same for $\tau_2(i+1)-\tau_2(i)$ with any $i\in \mathbb{N}$.

Now, let $I = \min\{i\in\mathbb{N}: \tau_2(i+1) - \tau_2(i) \geq \sqrt{\log n} \}$. Then, (\ref{eq:one step of small infection elimination}) tells us that
\begin{equation}\label{eq:small infec number of returns}
\P \left(\left.I \leq \sqrt{\log n} \,\right|\,|X_0|\leq \varepsilon_2 n, \, \mathcal{C}_n \right) \leq \sqrt{\log n} \cdot\exp \left(-\frac{1}{4}\sqrt{\log n} \right) \leq \exp \left(-\frac{1}{5}\sqrt{\log n} \right).
\end{equation}
This implies that $\tau_2(\sqrt{\log n}) \leq \sqrt{\log n}\cdot \sqrt{\log n} = \log n$ \textsf{whp}.
Moreover, since $X_{\tau_2(i)+1} = \zero$ with probability at least $(\log n)^{-1/3}$ by (\ref{eq:small inf extinction prob}),
\begin{equation*}
\P(X_{\tau_2(i)+1} \neq \zero ~\textnormal{ for all }i\leq \sqrt{\log n} ) 
\leq 
\P\left(\Bin\left( \sqrt{\log n},\, (\log n)^{-\frac{1}{3}} \right) =0 \right)
\leq
\exp\left(-(\log n)^{\frac{1}{6}} \right),
\end{equation*} 
which tells us that $X_{\tau_2(\sqrt{\log n})} = \zero$ \textsf{whp}. Along with the discussion in Section \ref{subsec:short surv step1}, we deduce that on $G_n\in \mathcal{C}_n$, the process started from $|X_0|\leq \varepsilon_2 n$ dies out in time $(C_2 +1) \log n$ \textsf{whp}.
Since $\P(G_n \in \mathcal{C}_n) = 1-o(1)$ by (\ref{eq:prop Cn}), the conclusion of Theorem \ref{thm:thm1}-2 follows. \qed

\subsection{Proof of Theorem \ref{thm:thm4}}\label{subsec:thm4}

The goal of this subsection is to establish Theorem \ref{thm:thm4} and Corollary \ref{cor:er}. To this end, we first show that Proposition \ref{prop:3gen} continues to hold when we replace the assumption $D\geq \theta+2$ by the \textit{$(\theta+2)$-core condition}. Recall that the $r$-core of $G$ is the largest induced subgraph of $G$ whose minimal degree is $r$.

%\begin{definition}[$r$-core]
%	Let $r$ be a positive integer and $G$ be a graph. The $r$-core of $G$ is the largest induced subgraph of $G$ whose minimal degree is $r$. If no such subgraph with size $>0$ exists, then we say that the $r$-core is an empty graph. 
%\end{definition}

Let $\mu$ be a given degree distribution  with $d=\E_{D\sim \mu}D$. Recall Proposition \ref{prop:kcore} where we defined $F_r(h)$ and $\rho_r(h)$,
%\begin{equation}
%\begin{split}
%F_r(h) &= \E\left[D_h\one_{\{ D_h\geq r \}} \right] = \sum_{l\geq r} \sum_{j\geq l} l\, \P( \textnormal{\Bin} (j, h)= l) \,\P(D=j);\\
%\rho_r(h) &= \P (D_h \geq r) =\sum_{l\geq r} \sum_{j\geq l} \P( \textnormal{\Bin} (j, h)= l) \,\P(D=j).
%\end{split}
%\end{equation}
 and saw  that the existence of $h\in (0,1)$ satisfying $dh^2 <F_r(h)$ implies the existence of the $r$-core of size $\approx \rho_r(\hat{h}) n$ in $G_n\sim \mathcal{G}(n,\mu)$ \textsf{whp}, as mentioned in the proposition. 
 
 Let $K_n$ be the $r$-core of $G_n$, and consider the event $\mathcal{F}^{*\theta}(m_1,m_2)$ for $K_n$. That is,
 \begin{equation*}
 K_n \in \mathcal{F}^{*\theta} (m_1,m_2) ~\textnormal{if and only if}~ \forall \,W\subset K_n \textnormal{ of size }m_1, ~|W^{*\theta}| \geq m_2.
 \end{equation*}
 Here, note that $W^{*\theta}$ is defined in terms of $K_n$, so that it contains vertices in $K_n$ that has at least $\theta$ neighbors in $W$. Under this setting, we have the following generalization of Proposition \ref{prop:3gen}.

\begin{proposition}\label{prop:3gen core}
	Let $\theta\geq 2$ be an integer and $\mu$ be a degree distribution that satisfies $\E_{D\sim \mu} D^6 = M<\infty$ and (\ref{eq:rcore condition}) with $r=\theta+2$. Let $\rho=\rho_{\theta+2}(\hat{h})$ with $\rho_{\theta+2}$ and $\hat{h}$ as in (\ref{eq:rcore eq def}) and the discussion below (\ref{eq:rcore condition}). Then, there exists a constant $\varepsilon_1(\theta, M)>0$ such that for $G_n\sim \mathcal{G}(n,\mu)$, its $(\theta+2)$-core $K_n$ is of size  $1-\frac{1}{100}\varepsilon_1\leq |K_n|/\rho n \leq 1+\frac{1}{100}\varepsilon_1 $ \textsf{whp}, and satisfies
	\begin{equation}\label{eq:prop3 core}
	\P \left(K_n \in \mathcal{F}^{*\theta} \left( \left\lceil(1-\varepsilon_1) \rho n\right\rceil , \lceil(1-0.9\varepsilon_1)\rho n\rceil \right)\right) 
	\geq
	1- o(1),
	\end{equation}
	for all large enough $n$, where $C>0$ is an absolute constant.
\end{proposition}

\begin{proof}
	Let $\varepsilon>0$ be a small constant that will be chosen later. Let $\{D_i\}_{i=1}^n$ be the degree sequence of $G_n$ generated by i.i.d. samples of $\mu$, and define the event $\mathcal{B}_n$ by 
	\begin{equation*}
	\mathcal{B}_n := 	\left\{\textnormal{for all }W\subset V \textnormal{ with } |W|=\varepsilon n,~  \sum_{i\in W}  D_i \leq 3 M^{\frac{1}{6}} \varepsilon^{ \frac{4}{5}} n \right\}.
	\end{equation*}
	Then, by Lemma \ref{lem:sum Di}, we have
	\begin{equation*}
	\P \left(G_n \in \mathcal{B}_n \right) \geq 1- \frac{C \varepsilon^{\frac{1}{5}}}{(\varepsilon n)^5}.
	\end{equation*}

	Now, given the degree sequence $\{D_i\}_{i=1}^n$ of $G_n$, we give an ordering to all half-edge clones as in Definition \ref{def:cutoffline}. Then we generate $K_n$, the $(\theta+2)$-core, using the cut-off line algorithm as follows:
	\begin{enumerate}
		\item[0.] Set the initial height of the cut-off line $H_0=1$. Let $L_0$ be the set of vertices $i$ with $D_i \leq \theta+1$, and $R_0$ be the set of half-edge clones attached to $L_0$. We call a vertex is \textit{light} if it has less than $\theta+2$ clones under the cut-off line.
		
		\item[1.] At step $t$,  we match the half-edge $a_t$ in $R_{t-1}$ that is chosen with respect to the prescribed order with the highest unmatched clone $a_t'$ (i.e., having the largest $B_j$ in the language of Definition \ref{def:cutoffline}) and set $H_t$ to be its height.
		
		\item[2.] Let $v_t$ be the vertex containing $a_t'$. Update $R_t$ by
		$$R_t = R_{t-1} \setminus \{a_t, a_t' \}.$$
		Moreover, add all the half-edges at $v_t$ below $H_t$ to $R_t$, if the number of them is less than $\theta+2$, i.e., $v_t$ becomes \textit{light} at step $t$. 
	\end{enumerate}
	
	\noindent Let $\mathcal{K}_n$ be the event that $K_n$, the $(\theta+2)$-core of $G_n$, is of size  $$1-\frac{\varepsilon}{100} \leq \frac{|K_n|}{\rho n} \leq 1+\frac{\varepsilon}{100}.$$ If $\mu$ satisfies (\ref{eq:rcore condition}) with $r=\theta+2$, then we have $\mathcal{K}_n$ \textsf{whp}, so that the process terminates in the sense that at least $\rho n$ vertices remain \textit{heavy} when $R_t$ becomes empty.  Let $H_{\textsc{core}}$ be the height of the cut-off line when $R_t$ is empty for the first time. Then the vertices $V_\textsc{core}$ remaining at that time forms $K_n$, and each $i\in V_\textsc{core}$ has $D_i' \geq \theta+2$ clones below $H_\textsc{core}$. 
	
	Let $U\subset V_{\textsc{core}}$ be any fixed subset of size $\lceil \varepsilon|V_\textsc{core}| \rceil$. Lemma \ref{lem:sum Di} implies that if $G_n \in \mathcal{B}_n$, regardless of the choice of $U$
	\begin{equation}\label{eq:sum Di core}
	\sum_{i\in U} D_i' \leq \sum_{i\in U}D_i \leq 3M^{\frac{1}{6}} \varepsilon^{\frac{4}{5}} n,
	\end{equation}
	by choosing $k=6$ and $\eta = \frac{1}{5}$ in (\ref{eq:sum Di}). Suppose that we start running the cut-off line algorithm again from the initial height $H_\textsc{core}$ to match all the remaining half-edges at $U$. Then, thanks to   (\ref{eq:sum Di core}), we can repeat the same argument as proof of Proposition \ref{prop:3gen} and obtain
	\begin{equation*}
	\P\left(\left.\left|V_\textsc{core} \setminus U^{*\theta} \right| \geq 0.9\varepsilon |V_\textsc{core}| \, \right|\, |V_\textsc{core}|,\, \mathcal{B}_n \cap \mathcal{K}_n \right)
	 \leq 
	 2\exp \left(-1.1\varepsilon \log\left(\frac{1}{\varepsilon} \right)|V_\textsc{core}| \right).
	\end{equation*}
	The only difference is to replace $d_0$ in Lemma \ref{lem:HT bound} and (\ref{eq:cutoff height bound for 2}) by $\hat{d}_0$, where
	\begin{equation*}
	\hat{d}_0 = F_{\theta+2} (\hat{h}),
	\end{equation*}
	with $F_{\theta+2}$ and $\hat{h}$ as in (\ref{eq:rcore eq def}) and the discussion below (\ref{eq:rcore condition}). Here, $\hat{d}_0 n/2$ corresponds to the expected total number of edges in $K_n$ (see \cite{jl07}, Theorem 2.3).
	
	Therefore, applying the union bound over all choices of $U$ deduces the conclusion for $\varepsilon_1$ satisfying (\ref{eq:epsilon1 def}), since the event $\mathcal{B}_n\cap \mathcal{K}_n $ happens \textsf{whp} for $G_n\sim \mathcal{G}(n,\mu)$. 
\end{proof}

We conclude the section by proving Theorem \ref{thm:thm4} and Corollary \ref{cor:er}.

\begin{proof}[Proof of Theorem \ref{thm:thm4}]
	Let $\rho = \rho_{\theta+2}(\hat{h})$, where $\rho_{\theta+2}$ and $\hat{h}$ are as in (\ref{eq:rcore eq def}) and the discussion below (\ref{eq:rcore condition}). To establish the first part of the theorem, we only focus on the infections inside $K_n$, the $(\theta+2)$-core of $G_n$, and repeat the argument in Section \ref{subsec:thm1-1} with
	\begin{equation*}
	p : = \frac{1-\frac{99}{100}\varepsilon_1}{1- \frac{95}{100}\varepsilon_1},
	\end{equation*}
	for $\varepsilon_1$ as in Proposition \ref{prop:3gen core}.
	Then, starting from $X_0\equiv\one$, $(1-\varepsilon_1)$ fraction of $K_n$ will remain infected for exponentially long time, and hence we deduce Theorem \ref{thm:thm4}-1.
	
	Suppose that we want to establish the second part of the theorem with the above $p$. Recalling Propositin \ref{prop:2gen} and the discussion in   Section \ref{subsec:thm1-2}, we required $\mu$ to satisfy
	\begin{equation*}
	\E_{D\sim \mu} D^{20/\delta} <\infty,
	\end{equation*}
	for $\delta = \frac{1}{3} (1-p)$. Here, we can see that $\frac{20}{\delta} = 60(1-p)^{-1} \leq 1500 \varepsilon_1^{-1}$. Based on the conditions for $\varepsilon_1$ given in (\ref{eq:epsilon1 def}), we can set $\varepsilon_1$ to be
	$$\varepsilon_1^{-1} = C_1' (\theta M)^{C_2} ,$$for some absolute constants $C_1', C_2>0$. Setting $C_1 = 1500C_1'$, $\mu$ falls into the regime where Proposition \ref{prop:2gen} works for $\delta = \frac{1}{3}(1-p)$, and hence we can repeat the argument in Section \ref{subsec:thm1-2} to deduce Theorem \ref{thm:thm4}-2.
\end{proof}

\begin{proof}[Proof of Corollary \ref{cor:er}]
	We can instead work with the configuration model $G_n \sim \mathcal{G}(n,\textnormal{Pois}(d))$ (\cite{k06}, Theorem 1.1).
	Then the first part of the Corollary can be proven analogously as above. Moreover,  since the Poisson distribution has all polynomial moments,  we can establish the second part for all $p$ as in Section \ref{subsec:thm1-2}.
\end{proof}

\section{Proof of Theorem \ref{thm:thm2}}\label{sec:thm2}

In this section, we establish Theorem \ref{thm:thm2}. Denoting the threhold-$\theta$ contact process by $(X_t)$ and random $(\theta+1)$-regular graph by $G_n\sim \mathcal{G}(n,\theta+1)$, the proof consists of the following observations.

\begin{enumerate}
	\item[O1.] For a cycle $C$ inside $G_n$, if $X_t(C)\equiv \zero$ at some time $t$, then $X_s(C)\equiv \zero$ for all $s\geq t$.
	
	\item[O2.] There exists a constant $\kappa(\theta)$ such that \textsf{whp}, $G_n \sim \mathcal{G}(n, \theta+1)$ can be covered by cycles of length at most $\kappa \log n$.
	
\end{enumerate}
 
\noindent The first observation is based on the definition of $(X_t)$. That is, if the entire cycle $C$ is healthy, then at each $v\in C$ we can only find at most $\theta-1$ infected neighbors of $v$, and hence  $v$ stays healthy in the next time step. The second one is obtained by the following proposition.

\begin{proposition}\label{prop:cycle}
	Let $d\geq 3$ be an integer. There exists a constant $\kappa(d)>0$ such that \textsf{whp} over $G_n\sim\mathcal{G}(n,d)$, we have that	for all $v\in V(G_n)$, there is a cycle $C_v \ni v$ of length at most $\kappa \log n$.
\end{proposition}

Assuming Proposition \ref{prop:cycle}, the two observations above easily imply Theorem \ref{thm:thm2}.

\begin{proof}[Proof of Theorem \ref{thm:thm2}]
	Suppose that $G_n\sim \mathcal{G}(n,\theta+1)$ can be covered by cycles of length at most $\kappa \log n$ for some constant $\kappa>0$. Set $a = 1+\kappa \log(\frac{1}{1-p})$.
	For a cycle $C$ of length at most $\kappa \log n$, we have 
	\begin{equation*}
	\P_{\textsc{tcp}}\left(X_{n^{a}}(C) \neq \zero \right) \leq \P\left(\Bin \left(n^{a},\, (1-p)^{\kappa \log n} \right)=0 \right) \leq e^{-n},
	\end{equation*}
	where $(1-p)^{\kappa \log n}$ corresponds to the probability that the entire cycle $C$ gets recovered at a single time step. 
	
	Let $C_1, \ldots, C_m$ be the cycles of length at most $\kappa \log n$ that cover $G_n$. Clearly $m\leq n$, and the above estimate tells us that
	\begin{equation*}
	\P_{\textsc{tcp}} \left( \bigcap_{l=1}^m \left\{X_{ln^{a}}(C_l) \equiv \zero  \right\} \right) \geq 1- ne^{-n},
	\end{equation*}
	where we devote each time interval $[(l-1)n^{a}+1, \, ln^{a} ]$ to eliminating infections inside $C_l$. Further, the  observation O1 tells us that if $X_{ln^a}(C_l)\equiv \zero$ for each $l$, then $X_{mn^a} \equiv \zero$. Therefore, the infection can survive at most $n^{a+1}$-time \textsf{whp}, given that $G_n$ can be covered with cycles of length at most $\kappa \log n$. Since the latter holds \textsf{whp} over the choice of $G_n\sim\mathcal{G}(n,\theta+1)$, the conclusion of Theorem \ref{thm:thm2} follows.	
\end{proof}

We conclude this section by showing Proposition \ref{prop:cycle}, which can be done similarly as \cite{bd82}.

\begin{proof}[Proof of Proposition \ref{prop:cycle}]
	Let $v$ be a fixed vertex in $G_n\sim \mathcal{G}(n,d)$ and set $L=2\log_{d-1} n$. We will show that the probability that the neighborhood $N(v,L)$ contains a cycle crossing $v$ is at least $1-o(n^{-1})$.  Then, the error $o(n^{-1})$ allows us to take a union bound over all vertices, hence implying Proposition \ref{prop:cycle}.
	
	Let $u_1, \ldots, u_d$ be the neighbors of $v$ (if $v$ has a self-loop, then we are done, so we assume that it has $d$ neighbors). Define the \textit{branch} of $u_i$ with respect to $v$ by
	\begin{center}
		$Br(u_i,v; l) := $\{$u\in V(G_n) : \exists$ a path of length $\leq l$ between $u$ and $u_i$ that does not cross $v$\}.
	\end{center}
	If there exist $i\neq j$ such that $Br(u_i, v; L_0+1) \cap Br(u_j, v; L_0+1) \neq \emptyset$ with $L_0 = \lfloor \log_{d-1} (n^{1/2}\log n)  \rfloor$, then we are done. Indeed, we claim that
	\begin{equation}\label{eq:branch intersection}
	\P \left( Br(u_1, v; L_0+1) \cap Br(u_2, v; L_0+1) \neq \emptyset \right) \geq 1 - o(n^{-1}).
	\end{equation}
	To see this, we first see that the branches have enough expansion with large probability. Suppose that we have explored $N(v,h)$, the vertices at distance $h$. At this point, the vertices at distance exactly $h$ from $v$ have unmatched half-edges which will be matched in the next step when we explore $N(v,h+1)$. Let $\partial N(v,h)$ be the collection of those unmatched half-edges, and let $\partial Br(i, h-1)$ be the ones who are attaced to vertices in $Br(u_i,v;h-1)$. Then,   as shown in equation (1) of \cite{bd82},
	\begin{equation}\label{eq:branch expansion}
	\P \left( |\partial Br(i,L_0)| \geq \frac{1}{2}(d-1)^{L_0 +1} ~~\textnormal{for all }i\in[d]  \right) \geq 1-o(n^{-1}).
	\end{equation}
	Further, given that $|\partial Br(1, L_0)|, |\partial Br(2,L_0)| \geq \frac{1}{2} n^{1/2} \log n$, the probability that there exist two half-edges, each from $Br(1,L_0)$ and $Br(2,L_0)$, are paired in the next exploration step is at least
	\begin{equation}\label{eq:matching L0}
	\P \left(\Bin\left(\frac{n^{1/2} \log n}{2}, \, \frac{\log n}{2dn^{1/2}} \right) \geq 1 \right) \geq 1-\exp\left(-\frac{(\log n)^2}{4d} \right).
	\end{equation}
	Therefore, combining (\ref{eq:branch expansion}) and (\ref{eq:matching L0}) implies (\ref{eq:branch intersection}), and hence we obtain the conclusion.
\end{proof}

%\section{Appendix}
%\subsection{Proof of Lemma \ref{lem:sum Di}}\label{subsec:lem Di}

\section*{Acknowledgement}

The author is grateful to Rick Durrett for introducing the problem and sharing his perspectives. He also thanks Rick Durrett, David Sivakoff, Souvik Dhara, Ankan Ganguly, Dan Han, Xiangying Huang, Yacoub Kureh and Matthew Wascher for fruitful discussions during the  workshop ``2019 AMS MRC: Stochastic spatial models.'' This material is based upon work supported by the National Science Foundation under Grant Number DMS 1641020 and by a Samsung Scholarship.

\bibliographystyle{plain}
\bibliography{TCPref}
\end{document}